\documentclass[11 pt]{amsart}
\usepackage{amsmath,amssymb,verbatim, enumerate}
\usepackage{tikz}

\usepackage{color}

\usepackage[enableskew]{youngtab}
\newcommand{\ds}{\displaystyle}
\newcommand{\mA}{\mathcal A}
\def\l{\lambda}
\def\m{\mu}
\def\n{\nu}

\def\o{\bar 1}
\def\t{\bar 2}
\def\th{\bar 3}
\def\f{\bar 4}

\def\cT{T^{\prec}}
\def\tcT{T^b}
\newcommand{\bT}{\mathbb T}
\definecolor{red}{rgb}{1,0,0}
\definecolor{blue}{rgb}{.2,.2,.8}

\newtheorem{theorem}{Theorem}[section]
\newtheorem{proposition}[theorem]{Proposition}

\newtheorem{corollary}[theorem]{Corollary}

\author{Cristina Ballantine
 \and Bill Hallahan
  }
 \title[Stability of Kronecker Coefficients]{Stability of coefficients in the Kronecker product of a hook and a rectangle}
\address{ Department of Mathematics and Computer Science, College of the Holy Cross, Worcester, MA, USA} 

\address{ cballant@holycross.edu }

\address{ wthall15@g.holycross.edu}

\begin{document}

\begin{abstract}

We use  recent work of Jonah Blasiak (2012) to prove a stability result for the coefficients in the Kronecker product of two Schur functions: one indexed by a hook partition and one indexed by a rectangle partition. We also give bounds for the size of the partition starting with  which the Kronecker coefficients are stable. Moreover, we show that once the bound is reached, no new   Schur functions appear in the decomposition of Kronecker product, thus allowing one to recover the decomposition from the smallest case in which the stability holds. 

 \end{abstract}
 \maketitle
\section{Introduction}
\label{sec:in}

Let $\chi^{\l}$ and $\chi^{\m}$ be the irreducible characters of  the
symmetric group on $n$ letters, $S_n$,  indexed by the partitions $\l$ and $\m$ of $n$. The
\emph{Kronecker product} $\chi^{\l}\chi^{\m}$ is defined by
$(\chi^{\l}\chi^{\m})(w)=\chi^{\l}(w)\chi^{\m}(w)$ for all $w\in S_n$. Thus,
$\chi^{\l}\chi^{\m}$ is the character that corresponds to the diagonal action of
$S_n$ on the tensor product of the irreducible representations indexed by $\l$ and
$\m$. Then, we have
$$\chi^{\l}\chi^{\m} =\sum_{\nu\vdash n}g(\l,\m,\n)\chi^{\nu},$$
where $g(\l,\m,\n)$ is the multiplicity of $\chi^{\nu}$ in $\chi^{\l}\chi^{\m}$.
Hence,  the numbers $g(\l,\m,\n)$ are non-negative integers. \medskip

 By means of the
Frobenius map one can define the Kronecker (internal) product on the Schur symmetric
functions by
$$s_{\l}\ast s_{\m}=\sum_{\n\vdash n} g(\l,\m,\n) s_{\n}.$$
A reasonable formula for decomposing the Kronecker product is unavailable, although
the problem has been studied since the early twentieth century. Some results exist in particular cases.
Lascoux \cite{la}, Remmel \cite{r}, Remmel and Whitehead \cite{rw} and Rosas
\cite{ro} derived closed formulas for  Kronecker products of Schur functions indexed
by two row shapes or hook shapes.  Dvir \cite{d} and Clausen and Meier
\cite{cm} have given for any $\lambda$ and $\mu$  a simple and precise description
for the maximum length of $\nu$ and the maximum size of $\nu_1$ whenever
$g(\l,\m,\n)$ is nonzero.  Bessenrodt and Kleshchev \cite{bk} have looked at the
problem of determining when the decomposition of the Kronecker product has one or
two constituents. Similarly,  positive combinatorial interpretations of the Kronecker coefficients $g(\l,\m,\n)$ exist only in particular cases: (i) if $\l$ and $\m$ are both  hooks \cite{r}; (ii) if $\l$ is a two row partition (with some conditions on the size of the first part)  \cite{bo-slc}; (iii) if $\lambda$ is a hook partition \cite{blasiak}; (iv) if $\l$ and $\m$ are both two row partition \cite{bms}. Recent years have seen a resurgence of the study of the Kronecker product motivated by application to geometric complexity theory and quantum information theory.  In the first seminal paper on the Kronecker product, Murnaghan \cite{m} observed the following stability property. Given  $\bar{\l}, \bar{\mu}$, and $\bar{\nu}$  partitions of $a,b$, and $c$ respectively, define $\l(n):=(n-a, \bar{\l})$, $\m(n):=(n-b, \bar{\m})$, and $\n(n):=(n-c,\bar{\n})$. Then, the Kronecker coefficient $g(\l(n),\m(n),\n(n))$ does not depend on $n$ for $n$ larger than some integer $N=N(\bar{\l}, \bar{\m},\bar{\n})$. We say that the Kronecker coefficients $g((n-a, \bar{\l}),(n-b, \bar{\m}),(n-c, \bar{\n}))$ \textit{stabilize} for $n \geq N$. Proofs of this stability property and lower bounds for $N$ were given by Brion \cite{brion} using algebraic geometry and Vallejo \cite{v1, v2} using combinatorics of the Young tableaux.  More generalized stability notions make sense. Pak and Panova recently proved  k-stability for Kronecker coefficients \cite{p-p}. Generalized stability has also been recently studied by Stembridge \cite{stem}.\medskip

In this article, using Blasiak's work \cite{blasiak}, we investigate the stability of the Kronecker coefficients $g(\l,\m,\n)$ when   $\l=(m^t)$ is a rectangle partition of  $n=mt$ and $\m=(n-d,1^d)$ is a hook partition of $n$. Specifically, if $\tilde{\n}^{(m)}$ is the partition obtained from $\n$ by adding a row of length $m$ and reordering the parts to form a partition, then, whenever $t\geq d+2$, we have $$g((m^t),(n-d,1^d),\nu)=g((m^{t+1}),(n-d+m,1^d),\tilde{\nu}^{(m)}).$$

Moreover, if $t\geq d+2$, all Schur functions appearing in the decomposition of $$s_{(n-d+m,1^d)}\ast s_{(m^{t+1})}$$
are of the form $\ds s_{\tilde{\nu}^{(m)}}$ for a partition $\nu$ such that $\ds s_{\nu}$ appears in the decomposition of $$s_{(n-d,1^d)}\ast s_{(m^{t})}.$$ Thus, if $n=m(d+2)$, one can completely recover the decomposition of the Kronecker product $$s_{(n-d+km,1^d)}\ast s_{(m^{d+2+k})}$$ from the decomposition of Kronecker product $$s_{(n-d,1^d)}\ast s_{(m^{d+2})}.$$

Our study of the particular case of the Kronecker product of a hook shape and a rectangular shape is motivated by its usefulness for the understanding of the quantum Hall effect \cite{stw}.


\section{Preliminaries and Notation} \label{sec:prelim}
In this section we set the notation and introduce some basic background about partitions and Schur functions, mostly following \cite{bo}. For details and proofs of the contents of this section see \cite{ma} or
\cite[Chapter 7]{st}.  Let $n$ be a non-negative integer. A \emph{partition} of  $n$
is a weakly decreasing sequence of non-negative integers,
$\l:=(\l_1,\l_2,\cdots,\l_{\ell})$, such that $|\l|=\sum \l_i=n$. We write $\l\vdash
n$ to mean $\l$ is a partition of $n$. The nonzero integers $\lambda_i$ are called
the \emph{parts} of $\l$. We identify a partition with its \emph{Young diagram},
i.e. the array of left-justified squares (boxes) with $\l_1$ boxes in the first row,
$\l_2$ boxes in the second row, and so on. The rows are arranged in matrix form from
top to bottom. By the box in position $(i,j)$ we mean the box  in the $i$-th row and
$j$-th column of $\l$. The \textit{size}  of $\l$ is $|\l|=\sum \l_i$. The \emph{length} of $\l$, $\ell(\l)$, is the number of rows
in the Young diagram. Given a partition $\l$, its \textit{conjugate} is the partition $\l'$ whose Young diagram has rows precisely the columns of $\l$. 

\begin{center}
\ \ \ \ \ { \yng(6,4,2,1,1)}\hspace*{3.5cm} { \yng(5,3,2,2,1,1)}

$\l=(6,4,2,1,1), \ \ \ \ell(\l)=5, \ \ \ |\l|=14$ \mbox{ \ \ and \ \  } $\l'=(5,3,2,2,1,1)$

\vskip 0.1in

 Fig. 1
\end{center}
Given two partitions $\l$ and $\m$, we write $\m\subseteq \l$ if and only if
$\ell(\m) \leq \ell(\l)$ and $\l_i\geq \m_i$ for $1\leq i\leq \ell(\m)$. If $\m
\subseteq \l$, we denote by $\l/\m$ the skew shape obtained by removing the boxes
corresponding to $\m$ from $\l$.

\begin{center}
{ \young(:::\hfil\hfil\hfil,:\hfil\hfil\hfil,:\hfil,\hfil,\hfil)} \vskip
0in $\l/\m$ where $ \l=(6,4,2,1,1)$  and  $\m=(3,1,1)$ \vskip 0.1in
 Fig.2
\end{center}
\medskip

A \emph{semi-standard Young tableau} (SSYT) \emph{of shape} $\l/\m$ is a filling of
the boxes of the skew shape $\l/\m$ with positive integers so that the numbers
weakly increase in each row from left to right and strictly increase in each column
from top to bottom. The \emph{type} of a SSYT $T$ is the sequence of non-negative
integers $(t_1,t_2,\ldots)$, where $t_i$ is the number of $i$'s in $T$.

\begin{center}
{ \young(:::2234,::1446,:1366,224)} \\is a SSYT of shape
$\l/\m=(7,6,5,3)/(3,2,1)$ and type $(2,4,2,4,0,3)$. \\  \vskip 0.1in Fig. 3
\end{center}
\vskip 0in Given a SSYT $T$ of shape $\l/\m$ and type $(t_1,t_2,\ldots)$, we define
its \emph{weight}, $w(T)$, to be the monomial obtained by replacing each $i$ in $T$
by $x_i$ and taking the product over all boxes, i.e.
$w(T)=x_1^{t_1}x_2^{t_2}\cdots$. For example, the weight of the SSYT in Fig. 3 is
$x_1^2x_2^4x_3^2x_4^4x_6^3$. The skew Schur function $s_{\l/\m}$ is defined
combinatorially by the formal power series
$$s_{\l/\m}= \sum_T w(T),$$
where the sum runs over all SSYTs of shape $\l/\m$. To obtain the usual Schur
function one sets $\m =\emptyset$.\medskip

 The space of homogeneous symmetric
functions of degree $n$ is denoted by $\Lambda^n$. A basis for this space is given
by the Schur functions $\{ s_\l\,|\, \lambda\vdash n\}$. The Hall inner product on
$\Lambda^n$ is denoted by $\langle \ , \ \rangle_{\Lambda^n}$ and it is defined by
$$\langle s_{\l},s_\m\rangle_{\Lambda^n}=\delta_{\l\m},$$
 where  $\delta_{\l\m}$  denotes the Kronecker delta.\medskip
 
  For a
positive integer $r$, let $p_r=x_1^r+x_2^r+\cdots $. Then $p_{\m}=p_{\m_1}p_{\m_2}
\cdots p_{\m_\ell(\m)}$ is the power symmetric function corresponding to the
partition $\m$ of $n$. If $CS_n$ denotes the space of class functions of $S_n$, then
the \emph{Frobenius characteristic map} $F: CS_n\rightarrow \Lambda^n$  is defined
by
$$F(\sigma)= \sum_{\m\vdash n} z^{-1}_{\m}\sigma(\m)p_{\m},$$
where $z_\m = 1^{m_1} \, m_1!\, 2^{m_2} \, m_2! \cdots n^{m_n} \, m_n!$ if
${\displaystyle \m=(1^{m_1}, 2^{m_2}, \ldots, n^{m_n})}$, i.e. $k$ is repeated $m_k$
times in $\m$, and $\sigma(\m)=\sigma(\omega)$ for an $\omega \in S_n$ of cycle type
$\m$. Note that $F$ is an isometry. If $\chi^\l$ is an irreducible character of
$S_n$ then, by the Murnaghan-Nakayama rule \cite[7.17.5]{st}, $F(\chi^\l)=s_\l$. \medskip

We define the Kronecker  product of  Schur 
functions by $$s_{\l}\ast s_{\m}=\sum_{\n\vdash n} g(\l,\m,\n) s_{\n},$$ where $g(\l,\m,\n)$ is the multiplicity of $\chi^{\nu}$ in $\chi^{\l}\chi^{\m}$.
 \medskip

  A \emph{lattice permutation} is a
sequence $a_1a_2\cdots a_n$ such that in any initial factor $a_1a_2\cdots a_j$, the
number of $i$'s is at least as great as the number of $(i+1)$'s for all $i$. For
example $11122321$ is a lattice permutation. The \emph{reverse reading word} of a
tableau is the sequence of entries of $T$ obtained by reading the entries from right
to left and top to bottom, starting with the first row.
\medskip

 \noindent {\bf
Example:} The reverse reading word of the tableau below  is $218653974$. 

\begin{center}
\Yvcentermath1$\young(::12,3568,479)$\\ \ \\ Fig. 4\end{center}


\section{Blasiak's Combinatorial Rule} \label{sec:blasiak}

In this section, following \cite{blasiak}, we give a brief description of the combinatorial interpretation of the Kronecker coefficients when one of the Schur functions is indexed by a hook shape. All partitions in this section are of size $n$. We write $\mu(d)$ for the partition $(n-d,d)$. First we introduce the necessary notation. \medskip

A \textit{word} is a sequence of letters from some totally ordered set called an \textit{alphabet}. The set $\{1, 2, \ldots\}$ is called the \textit{alphabet of unbarred (or ordinary) letters} and the set $\{\bar 1,\bar 2, \ldots\}$ is called the \textit{alphabet of barred  letters}. A colored \textit{word} is a word in the alphabet $\mA=\{1, 2, \ldots\}\cup \{\bar 1,\bar 2, \ldots\}$. We will need two orders on $\mA$. 

\begin{center}\textit{the natural order:} $\bar 1<1<\bar2<2<\cdots$ \end{center}

\begin{center}\textit{the small bar order:} $\bar 1\prec\bar2\prec\cdots \prec 1 \prec2\prec\cdots$ \end{center}\medskip

A \textit{semistandard colored tableau} for any of the above orders on $\mA$ is a tableau with entries in $\mA$ such that: (i) unbarred letters increase weakly from left to right in each row and strictly from top to bottom in each column; (ii) barred letters increase strictly from left to right in each row and weakly from top to bottom in each column. \medskip

Given a semistandard colored tableau for any one of the above orders, one can convert it to a semistandard colored tableau for the other order using Jeu-de-Taquin moves. Since the description of the moves is rather involved, we use Example 2.17 in \cite{blasiak} to show the conversion of a semistandard colored tableau for the small bar order to a semistandard colored tableau for the natural bar order. Converting from the natural order to the small bar ordered is obtained by simply reversing the steps.\medskip 

\noindent \young(\o\t\th1,\o\th\f2,\t113,124,35) \hspace*{-.06in} $\leftrightarrow$ \hspace*{-.06in} \young(\o\t\th1,\o\th12,\t13\f,124,35) \hspace*{-.06in}$\leftrightarrow$  \hspace*{-.06in} \young(\o\t\th1,\o112,\t23\f,1\th4,35) \hspace*{-.06in}$\leftrightarrow$  \hspace*{-.06in} \young(\o\t11,\o12\th,\t23\f,1\th4,35) \ \hspace*{-.06in}$\leftrightarrow$  \hspace*{-.06in} \young(\o\t11,\o12\th,123\f,\t\th4,35)  \hspace*{-.06in}$\leftrightarrow$  \hspace*{-.06in}\young(\o111,\o\t2\th,123\f,\t\th4,35)
\begin{center}Fig. 5\\ Blasiak's example 2.17 \end{center}
\bigskip

In step 1, perform Jeu-de-Taquin on $\f$; in step 2, perform Jeu-de-Taquin on the lower $\th$; and so on. \medskip

The \textit{content} of a colored tableau $T$ is $c=(c_1, c_2, \ldots)$, where $c_i$ is the number of $i$ and $\bar{i}$ in $T$. The \textit{total color} of $T$ is the number of barred letters in $T$.\medskip

Let $T^<$ be a colored tableau for natural order $<$ and let $\cT$ be the tableau obtained from $T^<$ by converting to the small bar order $\prec$. Let $\tcT$ be the tableau of barred letters in $\cT$. Denote by $\bT$ the tableau obtained by placing $\cT/\tcT$ above and to the right of $(\tcT)'$ so that the NW corner of $(\tcT)'$ touches the SE corner of $\cT/\tcT$, and removing the bars from the letters of $(\tcT)'$. \medskip

In Figure 6, we show an example of the tableaux $T^<, \cT, \tcT$, and $\bT$, where $T^<$ is the tableau on the right in the example above. 
\begin{center}

 \young(\o111,\o\t2\th,123\f,\t\th4,35) \hspace{1cm} \young(\o\t\th1,\o\th\f2,\t113,124,35) \hspace{1cm}  \young(\o\t\th,\o\th\f,\t)\hspace{1cm} \young(::::::1,::::::2,::::113,:::124,:::35,112,23,34)\vspace{.1in}

\hspace{-1.5cm} $T^<$ \hspace{2.2cm} $\cT$ \hspace{2.2cm} $\tcT$ \hspace{2.2cm} $\bT$\\ \vskip 0.1in

Fig. 6
\end{center}

\medskip

In the rule below, we refer to a semistandard Young tableau for the natural order $<$ as a \textit{colored tableau}. 
A colored tableau $T^<$ is called \textit{Yamanouchi} if the reverse reading word of $\mathbb T$ is a lattice permutation. \medskip

\noindent \textbf{Blasiak's combinatorial rule:} The Kronecker coefficient $g(\l, \m(d),\nu)$ equals the number of Yamanouchi  colored tableaux of shape $\nu$,  content $\l$, and total color $d$, in which the box in the SW corner has an unbarred letter. 

\section{On the Kronecker Product of a Hook and a Rectangle} \label{sec:properties}

In this section we collect several properties of the partitions indexing Schur functions that appear in the Kronecker Product of the Schur function associated with a hook partition and the Schur function associated with a rectangular partition. For the remainder of the article, $T^<$ will denote a semistandard Young tableau for the natural order $<$ and $\cT, \tcT, \bT$ will denote the tableaux associated with $T^<$ defined in the previous section.\medskip

In the next theorem, we prove an important property of tableaux counted by Blasiak's rule when $\l$ is a rectangle partition $(m^t)$. For this proposition, the fact that the SW corner of the tableau is unbarred is irrelevant. \medskip

\begin{theorem} \label{determined} Given a Yamanouchi colored tableaux $T^<$ of shape $\nu$, content $(m^t)$, and total color $d$, the shape of $\tcT$ completely determines the tableau $\tcT$. 

\end{theorem}

\begin{proof} By the definition of a semistandard colored tableau, $\bT$ must be an ordinary SSYT of type $(m^t)$. Moreover, its reverse reading word is a lattice permutation. Let $\eta$ be the shape of $\tcT$. In $\bT$, the shape of $(\tcT)'$ is $\eta'$.  Since there are exactly $m$  of each letter appearing in $\bT$, the reverse reading word of $\bT$ must end in $t$,  the highest available letter, and therefore the last row of $\bT$, which has length $\eta'_{\eta_1}$ is filled with $t$. Since the type of $\bT$ is a rectangle, the second to last row in $\bT$ must start with $t-1$ and thus is can   only be filled with the  letters $t-1$ and $t$. By the semistandard condition, the first $\eta'_{\eta_1}$ entries in this row must be $t-1$. Moreover, by the lattice permutation condition and the fact that the type is a rectangle, there cannot be more than  $\eta'_{\eta_1}$ letters $t-1$ in the second to last row in $\bT$. Continuing in this way, we see that  the content of $\tcT$ is determined. In fact, each row of $\tcT$ is filled in decreasing order from right to left with the letters $\bar t, \overline{t-1}, \overline{t-2}, \ldots$. 
\end{proof}

The next observation is  an immediate consequence of Theorem \ref{determined}. 

\begin{corollary} \label{cor-determined} Let $T^<$ be  a Yamanouchi colored tableaux  of shape $\nu$, content $(m^t)$, and total color $d$ and let $\eta$ be the shape of $\tcT$. The barred letters in $\tcT$ are precisely $\bar t, \overline{t-1}, \ldots, \overline{t-\eta_1+1}$. Moreover, the partition obtained by reading the parts of the content of $\tcT$ (i.e., of only the barred letters) in reverse order is precisely $\eta'$. 

\end{corollary}

In the remainder of this section we prove two  properties of colored Yamanouchi tableaux of shape $\nu$, content $(m^t)$, and total color $d$, in the case when $t= d+w$, where $w \geq 1$.


%

\begin{proposition}  \label{barred-lengtheta} Let $t=d+w$ and assume that $w\geq 1$. Let $T^<$ be a colored Yamanouchi tableaux of shape $\nu$, content $(m^t)$, and total color $d$. Then for every letter $\bar s$ in $T^<$ (and thus in $\tcT$), we have $s\geq \ell(\eta)+w$, where $\eta$ be the shape of $\tcT$.

\end{proposition}

\begin{proof} Let $\bar s \in \tcT$. If $s=t$, then, since $\eta\vdash d$, we have  $s=d+w\geq \ell(\eta)+w$. Suppose now that $s<t$. By Corollary \ref{cor-determined}, the letter ${\bar t}$ appears   in $\tcT$ exactly $\ell(\eta)$ times. Thus, there are exactly $d-\ell(\eta)$ barred letters less than $\bar t$ in $\tcT$. Also, by Corollary \ref{cor-determined}, if $\bar s$ is in $\tcT$, then so are $\overline{s+1}, \overline{s+2}, \ldots, \bar t$. Thus there are at least $t-s$ labels less than $\bar t$ in $\tcT$ and we must have $t-s\leq d-\ell(\eta)$. Thereofore, $s \geq \ell(\eta)+ t-d=\ell(\eta)+w$. 

\end{proof} 

\begin{corollary}\label{cor-barred-lengtheta} Suppose $T^<$ and $\eta$  are as in Proposition \ref{barred-lengtheta}, and let $j=\max\{s \mid \bar s \mbox{\ \  is not in } \tcT\}$. Then $j \geq \ell(\eta)+w-1$. 

\end{corollary}

\section{Stability of the   Kronecker coefficients} \label{sec:stability}

In this section we state and prove our main result, a stability property for the Kronecker coefficients in the particular case when one shape is a hook and the other is a rectangle. Moreover, we give a  bound for the size of the partition starting with which the Kronecker coefficients are stable. We also show that once the stability bound is reached, no new Schur function appear in the decomposition of the Kronecker product. 
 \medskip

We first show that for $n=(d+w)m$, with $w \geq 2$, if $g((m^{d+w}),(n-d,1^d), \nu)>0$, then $\nu$ contains at least $w-1$ rows of length $m$. 

\begin{theorem}\label{rowlength-m} Let $t=d+w$ and assume that $w\geq 2$.  Let $T^<$ be a colored Yamanouchi tableaux of shape $\nu$, content $(m^t)$, and total color $d$. Let $\eta$ be the shape of $\tcT$. Then,  for  $1\leq p\leq w-1$, we have $\ds \nu_{\ell(\eta)+p}=m$. 
\end{theorem}

\begin{proof} Let $j=\max\{s \mid \bar s \mbox{\ \  is not in } \tcT\}$. Then, by Corollary \ref{cor-barred-lengtheta}, $j \geq \ell(\eta)+w -1$. Since $w\geq 2$, we have $j\geq \ell(\eta)+1$. Thus, all $m$ letters $\ell(\eta)+1$ are unbarred in $T^<$. By the lattice permutation condition, letter $\ell(\eta)+1$ can appear in $\cT$ only in row $\ell(\eta)+1$ or below it. By the semistandard condition, letter $\ell(\eta)+1$ must appear in $\cT$ in $m$ different columns. Therefore, the length of row $\ell(\eta)+1$ in $\cT$ must be at least $m$, i.e., $\ds \nu_{\ell(\eta)+1}\geq m$. \medskip

Let $S$ be the tableau consisting of the first $\ell(\eta)$ rows of $\cT/\tcT$ and let $\ds \xi=(\xi_1, \xi_2, \ldots, \xi_{\ell(\eta)})$ be the content of $S$. (Note that some of the last entries in $\xi$ could be $0$.) By the semistandard condition, row $\ell(\eta)+1$ in $\cT$ must contain  the letter $1$ exactly $m-\xi_1$ times. By the lattice permutation condition, for each $i=2, 3, \ldots, \ell(\eta)+1$, row $\ell(\eta)+1$ in $\cT$ contains label $i$ at most $\xi_{i-1}-\xi_i$ times. Recall that $\ds \xi_{\ell(\eta)+1}=0$. No letter larger than $\ell(\eta)+1$ can appear in this row.  Therefore, we have 

$$ \nu_{\ell(\eta)+1}\leq (m-\xi_1)+(\xi_1-\xi_2)+(\xi_2-\xi_3)+\cdots + (\xi_{\ell(\eta)}-\xi_{\ell(\eta)+1})=m.$$ Thus, $\ds \nu_{\ell(\eta)+1}=m$. 
Then, if $p \geq 1$, we must have $\ds \nu_{\ell(\eta)+p}\leq m$. \medskip

Now, since $j \geq \ell(\eta)+w -1$, if $p \leq w-1$, then $j \geq \ell(\eta)+p$. In $\cT$, all $m$ letters $\ell(\eta)+p$ are unbarred. By the lattice permutation condition, letter $\ell(\eta)+p$ can appear in $\cT$ only in row $\ell(\eta)+p$ or below it. By the semistandard condition, letter $\ell(\eta)+p$ must appear in $\cT$ in $m$ different columns. Therefore,  $\ds \nu_{\ell(\eta)+p}\geq m$. 
\end{proof}

Fix the positive integers $m$ and $d$ and suppose that $t=d+w$ with $w\geq 1$. Let $A_{w}^{\prec}$ be the collection of Yamanouchi semistandard colored tableaux $\cT$ (for the \textit{small bar order}) of  content $(m^t)$, and total color $d$ (and any shape). \medskip

\begin{proposition} \label{determined-rows} Let $\cT$ be a tableau in $A_{w}^{\prec}$. Then,   the shape $\eta$ of $\tcT$ and the first $\ell(\eta)$ rows of $\cT$ completely determine the filling of the rows $\ell(\eta)+1, \ell(\eta)+2, \ldots, \ell(\eta)+w-1$. \end{proposition}

\begin{proof} As before, let $S$ be the tableau consisting of the first $\ell(\eta)$ rows of $\cT/\tcT$ and denote by  $\ds \xi=(\xi_1, \xi_2, \ldots, \xi_{\ell(\eta)})$ be the content of $S$. Let $\xi_0=m$. By the proof of Theorem \ref{rowlength-m}, for each $1\leq i \leq \ell(\eta)+1$,  the letter $i$ appears in row $\ell(\eta)+1$ of $\cT$ exactly $\xi_{i-1}-\xi_i$ times; and these are precisely the letters in row $\ell(\eta)+1$. By the same argument, for $ 2\leq k \leq w-1$, each box in  row $\ell(\eta)+k$,  is filled with the letter obtained by adding $1$ to the letter directly above it. 

\end{proof}

Suppose $w\geq 1$. We define a map $\varphi: A_w^{\prec} \to A_{w+1}^{\prec}$ as follows. Given $\cT\in A_w^{\prec}$ with the shape of $\tcT$ equal to $\eta$, then $\varphi(\cT)$ is the tableau obtained from $\cT$ by performing the steps below. \medskip

(i) Increase each barred letter of $\tcT$ by $1$. \medskip

(ii) Keep the first $\ell(\eta)+w-1$ rows of $\cT/\tcT$ unchanged. \medskip

(iii) Insert a row of length $m$ after row $\ell(\eta)+w-1$. If $w\geq 2$, each  box of the new row is filled with the letter obtained by adding $1$ to the letter in the corresponding box of row $\ell(\eta)+w-1$. If $w=1$,  the new row is exactly row $\ell(\eta)+1$ described in the proof of Proposition   \ref{determined-rows}. \medskip

(iv) increase each letter in the remaining rows by $1$. \medskip 

Obviously $\varphi(\cT)$ has total color $d$. It is straightforward to check that $\varphi(\cT)$ has content $(m^{t+1})$ and thus it belongs to $A_{w+1}^{\prec}$.

\begin{theorem}\label{small-bar-bijection} If $w \geq 1$, the map $\varphi: A_w^{\prec} \to A_{w+1}^{\prec}$ defined above is a bijection. 

\end{theorem}

\begin{proof} The inverse of $\varphi$ is the function defined on colored Yamanouchi semistandard tableaux in $A_{w+1}^{\prec}$ by reversing the steps above. By Theorem \ref{rowlength-m}, reversing step (iii), i.e., deleting a row of length $m$, is always possible. 

\end{proof}

For the rest of the article, if $T^<$ is a colored tableau in the natural order, by  $(T^<)^{\prec}$ we mean the tableau obtained from $T^<$ by converting to the small bar order. Similarly,  if $\cT$ is a colored tableau in the small bar order, by  $(\cT)^{<}$ we mean the tableau obtained from $\cT$ by converting to the natural order.\medskip

 Again, fix positive integers $m$ and $d$ and let $t=d+w$, with $w\geq 1$. We denote by $B_w^<$ the   collection of Yamanouchi semistandard colored tableaux $T^<$ (for the \textit{natural order}) of  content $(m^t)$, and total color $d$ (and any shape). \medskip
 
 We define a map $\psi:B_w^<\to B_{w+1}^<$ as follows. Let $T^< \in B_w^<$. We have $(T^<)^{\prec}\in A_w^{\prec}$. Then, $\psi(T^<)= \varphi((T^<)^{\prec})^<$. Thus, $T^<$ is converted to the small bar order, mapped by $\varphi$ to $A_{w+1}^{\prec}$, and then converted back to the natural order. 
 
 \begin{theorem} \label{natural-bijection} If $w\geq 2$, the map $\psi:B_w^<\to B_{w+1}^<$ defined above is a bijection. 
 
 \end{theorem} 

\begin{proof} The rigorous proof of this theorem is  technical. Here we present a sketch of the proof. 
 The main idea is to show that, given a tableau $\cT\in A_w^{\prec}$, the SW corner in $(\cT)^<$ is barred if and only iff the SW corner of $\varphi(\cT)^<$ is barred. Since, by Theorem \ref{small-bar-bijection},  $\varphi$ is a bijection, this proves the statement of the theorem.  \medskip
 
 Let $\cT\in A_w^{\prec}$. As before, $T^b$ is the tableau of barred letters in $\cT$ and we denote its shape by $\eta$. To simplify the discussion, we refer to the first $\ell(\eta)$ rows of $\cT/T^b$ as $T_1$. We refer to  the next $w-1$ rows of $\cT/T^b$ (all of length $m$) as $T_2$, and we refer to the remaining rows as $T_3$. Thus $\varphi(\cT)$ is obtained by increasing each letter in $T^b$ by one, keeping  $T_1$ and $T_2$ unchanged, adding one to each letter of $T_3$, and inserting a row of length $m$ between $T_2$ and $T_3$ whose filling is obtained by adding one to the letter in each box directly above it. We refer to this row as $R$. \medskip
 
When converting both $\cT$ and $\varphi(\cT)$ to the natural order, all Jeu-de-Taquin moves will be essentially the same. We explain this below. \medskip

By the construction of $\varphi(\cT)$, because $\varphi(\cT)^b$ is obtained by adding one to each letter in $T^b$ and $T_1$, $T_2$ remain unchanged, the Jeu-de-Taquin moves that stay in $T_1$ and $T_2$, i.e., in the first $\ell(\eta)+w-1$ rows, are \textit{exactly} the same in $\cT$ and $\varphi(\cT)$. \medskip

Now, suppose a barred letter $\bar b$ has arrived by Jeu-de-Taquin moves in row $\ell(\eta)+w-1$ of $\cT$, i.e., the last row of $T_2$. Then $\overline{b+1}$ has arrived in the same place in $\varphi(\cT)$. We have the following situation. On the left, we are in $\cT$ and the line shows the delimitation between $T_2$ and $T_3$. On the right, we are in $\varphi(\cT)$, the first line is the delimitation between the \textit{old} $T_2$ (up to row $\ell(\eta)+w-1$) and the inserted row $R$. The second line is the delimitation between $R$ and $T_3$. 

\begin{center}

\hspace{-2cm} in $\cT$: \hspace{5cm} in $\varphi(\cT)$:\vspace{-.1cm}

\begin{figure}[ht]
\begin{picture}(150,0)
\put(50,0){$\bar b$}\put(60,0){$a$}
\put(30,-5){\line(1,0){50}}
\put(50,-15){$c$} \put(60,-15){$d$}
\end{picture}
\begin{picture}(200,10)
\put(50,0){$\overline{b+1}$}\put(82,0){$a$}
\put(30,-5){\line(1,0){80}}
\put(55,-15){x}\put(78,-15){$a+1$}
\put(30,-20){\line(1,0){80}}
\put(50,-30){$c+1$} \put(78,-30){$d+1$}
\end{picture}

\end{figure}
\vspace{.4in}

\hspace{-2cm} Fig. 7(a) \hspace{5cm} Fig. 7(b)
\end{center}\vspace{.1cm}

Note that if $\overline{b+1}$ arrived by Jeu-de-Taquin moves directly above $x$, then $x\neq a+1$. (If $x=a+1$, originally there was a letter $a$ directly above $x$. Then $\overline{b+1}$ could not have been moved down from directly above this $a$. If $\overline{b+1}$ arrived to the position in Fig 7(b) by a horizontal  switch, this means that row $\ell(\eta)+w-1$ was filed with the letter $a$ from the beginning until (at least) the letter $a$ depicted in Fig. 7(b). However, in this case, $\overline{b+1}$ could not have been moved down to row $\ell(\eta)+w-1$ from the row above it.) \medskip

The next Jeu-de-Taquin move in $\varphi(\cT)$ is to switch $\overline{b+1}$ and $x$ (this is an additional move to the moves in $\cT$). After this move, we have:\newpage

\begin{center}

\hspace{-2cm} in $\cT$: \hspace{5cm} in $\varphi(\cT)$:

\begin{figure}[ht]
\begin{picture}(150,0)
\put(50,0){$\bar b$}\put(60,0){$a$}
\put(30,-5){\line(1,0){50}}
\put(50,-15){$c$} \put(60,-15){$d$}
\end{picture}
\begin{picture}(200,10)
\put(50,-16){$\overline{b+1}$}\put(82,0){$a$}
\put(30,-5){\line(1,0){80}}
\put(55,0){x}\put(78,-15){$a+1$}
\put(30,-20){\line(1,0){80}}
\put(50,-30){$c+1$} \put(78,-30){$d+1$}
\end{picture}

\end{figure}
\vspace{.4in}

\hspace{-2cm} Fig. 8(a) \hspace{5cm} Fig. 8(b)
\end{center}\vspace{.1cm}

To summarize, in $\varphi(\cT)$, Jeu-de-Taquin moves of barred letters in  row $\ell(\eta)+w-1$ always go \textit{down} to row $R$. \medskip

Since $a<c$ if and only if $a+1<c+1$, the moves of barred letters in row $\ell(\eta)+w-1$ in $\cT$, respectively $R$ in $\varphi(\cT)$, and the moves in $T_3$ are \textit{exactly} the same in $\cT$ and $\varphi(\cT)$.
\medskip

Therefore, in tableaux $(\cT)^<$ and $(\varphi(\cT))^<$ the SW corner is either barred in both or unbarred in both.

\end{proof}

Theorems \ref{small-bar-bijection} and \ref{natural-bijection} together with Blasiak's combinatorial rule lead to our main theorem. We use the notation $\tilde{\nu}^{(m)}$ to mean the partition obtained from $\nu$ by adding a part of length $m$ and rearranging the parts to form a partition. 

\medskip

\begin{theorem}\label{main} Fix  integers $m\geq 1$ and $d\geq 0$. Then, whenever $t\geq d+2$, we have  $$g((m^t),(n-d,1^d),\nu)=g((m^{t+1}),(n-d+m,1^d),\tilde{\nu}^{(m)}),$$ where $n=mt$.  Moreover, if $\ds g((m^{t+1}),(n-d+m,1^d),\gamma)>0$, then there exists $\nu\vdash mt$ such that $\gamma=\tilde{\nu}^{(m)}$. 
\end{theorem}


\section{Final remarks} 

As stated in Theorem, \ref{main}, the stability property for the Kronecker product of a Schur function indexed by a hook partition and one indexed by a rectangular partition proved in this article is much stronger than usual stability properties. Starting with the stability bound, as one increases the size of the partitions by $m$, no new partitions are introduced. Therefore, if $n=m(d+w)$ and $w\geq 2$,  one can completely recover the decomposition of the Kronecker product $$\ds s_{(n-d,1^d)}\ast s_{(m^{d+w})}$$ from the Kronecker product  $$\ds s_{(m(d+2)-d,1^d)}\ast s_{(m^{d+2})}.$$ \medskip

We note that the bound $n=tm$, where $t=d+w$ with $w\geq 2$, starting with which the stability of Kronecker coefficients holds, i.e., $$g((m^t),(n-d,1^d),\nu)=g((m^{t+1}),(n-d+m,1^d),\tilde{\nu}^{(m)}),$$ is nearly sharp. If $w=1$ we believe that the stability still holds and the method of proof should be  similar to the case $w\geq 2$. \medskip

If $n=(d-1)m$, we have verified using Maple that $$g((3^3),(5,1^4),(5,2,1,1))=2$$ while $$g((3^4),(8,1^4), (5,3,2,1,1))=3.$$ Therefore the stability property fails in this case. We are uncertain what happens if $w=0$.

\section{Acknowledgements} 

The second author would like to thank Dr. Dan Kennedy for providing the funds to support his research on this problem during the summer of 2014. 



\begin{thebibliography}{99}
\bibitem[BK]{bk}  Bessenrodt, C., Keleshchev, A.; ``On Kronecker products of Complex Representations
of the Symmetric and Alternating groups'', \emph{Pacific J. of Math}. Vol {\bf 190} 2,
1999.
\bibitem[BO-1]{bo} Ballantine, C., Orellana, R.; ``On the Kronecker
Product $s_{(n-p,p)}\ast s_{\l}$'',  \emph{Electron. J. Combin.} 12 (2005), Research Paper 28, 26 pp.

\bibitem[BO-2]{bo-slc}  Ballantine, C., Orellana, R.; ``A combinatorial interpretation for the coefficients in the Kronecker product $s_{(n-p,p)}\ast s_{\lambda}$'', S\'em. Lothar. Combin. 54A (2005/07), Art. B54Af, 29 pp. 

\bibitem[B]{blasiak} Blasiak, J: ``Kronecker coefficients for one hook shape'',  	arXiv:1209.2018 [math.CO], Sept. 2012.

\bibitem[BMS]{bms} Blasiak, J., Mulmuley, K., Sohoni, M.; ``Geometric Complexity Theory IV: nonstandard quantum group for the Kronecker problem'', arXiv:cs/0703110v4 [cs.CC], June 2013.

\bibitem[Br]{brion} Brion, M.; ``Stable properties of plethysm: on two conjectures of Foulkes'',
\emph{Manuscripta Math.}
80
(1993), 347?371.

\bibitem[CM]{cm} Clausen, M., Meier, H.; ``Extreme irreduzible Konstituenten in Tensordarstelhungen
symmetrischer Gruppen'', \emph{Bayreuther Math. Schriften.}  {\bf 45} (1993), 1-17.
\bibitem[D]{d} Dvir, Y.; ``On the Kronecker product of $S_n$ characters'', \emph{J. Algebra}
{\bf 154} (1993), 125-140.
\bibitem[La]{la} Lascoux, A.; ``Produit de Kronecker des representations du group
symmetrique", \emph{Lecture Notes in Mathematics} 1980, {\bf 795}, Springer Verlag
pp. 319-329.
\bibitem[Ma]{ma} Macdonal, I.G.; "Symmetric Functions and Hall polynomials", sec. ed.
\emph{Oxford University Press}, 1995.
\bibitem[M]{m} Murnaghan, F.D.; ``The Analysis of the Kronecker Product of
Irreducible Representation of the Symmetric Group'', \emph{American Journal of
Mathematics}, Vol. 60. No. 3, 761-784 (1938).
\bibitem[PP]{p-p} Pak, I., Panova, G.; ``Bounds on the Kronecker coefficients'', arXiv:1406.2988v2 [math.CO] June 2014. 
\bibitem[R]{r} Remmel, J.; ``A Formula for the Kronecker Products of Schur
Functions of Hook Shapes'', \emph{Jornal of Algebra} {\bf 120}, 100-118 (1989).
\bibitem[RWd]{rw} Remmel, J.,Whitehead, T.; ``On the Kronecker product of Schur
functions of two row shapes'', \emph{Bull. Belg. Math. Soc.} 1, 1994, pp. 649-683.
\bibitem[Ro]{ro} Rosas, M. H.; "The Kronecker
product of Schur functions indexed by two-row shapes or hook shapes", \emph{J.
Algebraic Combin.} 14 (2001), no. 2, 153--173.

\bibitem[STW]{stw} Scharf, T., Thibon J.-Y., and Wybourne, B.G.; "Powers of the Vandermonde determinant and the quantum Hall effect", \emph{J. Phys. A: Math. Gen.} 27 (1994) 4211-4219.

\bibitem[S]{st} Stanley, R.P.; "Enumerative Combinatorics" Vol 2,
\emph{Cambridge Univ. Press}, 1999.

\bibitem[St]{stem}Stembridge J.;
''Generalized stability of Kronecker coefficients'', preprint, August 2014. Available on http://www.math.lsa.umich.edu/
jrs/papers


\bibitem[V1]{v1} Vallejo, E.; ``Stability of the Kronecker products or irreducible characters
of the symmetric group'' \emph{The Elec. J. of Comb.} {\bf 6} (1999).

\bibitem[V2]{v2} Vallejo E.;
''Stability of Kronecker coefficients via discrete tomography'',  arXiv:1408.6219
\end{thebibliography}

\end{document}